\documentclass{amsart}
\usepackage{amsfonts}
\usepackage{amsmath}
\usepackage{amssymb}
\usepackage{amsthm}

\newtheorem{theorem}{Theorem}
\newtheorem{prop}{Proposition}
\newtheorem{lemma}{Lemma}
\newtheorem{rem}{Remark}

\newtheorem{fact}{Fact}

\begin{document}

\title{Geometric version of Wigner's theorem for Hilbert Grassmannians}
\author{Mark Pankov}
\subjclass[2000]{}
\address{Faculty of Mathematics and Computer Science, 
University of Warmia and Mazury, S{\l}oneczna 54, Olsztyn, Poland}
\email{pankov@matman.uwm.edu.pl}

\maketitle

\begin{abstract}
We show that the transformations of Grassmannians (of complex Hilbert spaces) 
induced by linear or conjugate-linear isometries can be characterized as 
transformations preserving some of principal angles
(corresponding to the orthogonality, adjacency and ortho-adjacency relations).
\end{abstract}

\section{Introduction}
Let $H$ be a complex Hilbert space of dimension not less than $3$. 
There is a natural one-to-one correspondence between closed subspaces of $H$ and 
projections, i.e. self-adjoint idempotents in the algebra of all bounded linear operators on $H$. 
Denote by ${\mathcal G}_{k}(H)$ the Grassmannian formed by all $k$-dimensional subspaces of $H$, 
in other words, all projections of rank $k$. 
We can assume that $\dim H\ge 2k$ (by duality). 

In quantum mechanics, projections of rank $1$ are identified with so-called {\it pure states}.
The {\it transition probability} ${\rm Tr}(P,P')$ for two pure states $P, P'\in {\mathcal G}_{1}(H)$ is equal to 
$|\langle x,x'\rangle|^2$, where $x\in P$ and $x'\in P'$ are unit vectors. 
By classical Wigner's theorem, every bijective transformation of ${\mathcal G}_{1}(H)$ preserving the transition probability 
is induced by an unitary or anti-unitary operator on $H$.
 This statement plays an important role in the mathematical foundations of quantum mechanics.
 
 It was observed by Uhlhorn in \cite{Uhlhorn} that the same holds for
 bijective transformations of ${\mathcal G}_{1}(H)$ preserving the orthogonality relations in both directions
 (this fact is a simple consequence of the Fundamental Theorem of Projective Geometry). 
 Since the transition probability is zero if and only if the corresponding pure states are orthogonal,
classical Wigner's  theorem is contained in Uhlhorn's statement.
However, there is a non-bijective version of Wigner's theorem which states that 
every (not necessarily bijective) transformation of ${\mathcal G}_{1}(H)$ preserving the transition probability is induced by 
a linear or conjugate-linear isometry on $H$ (see, for example, \cite{Geher1}).

The {\it principal angles} $0\le \theta_{1}\le\dots\le \theta_{k}\le \pi/2$ 
between $k$-dimensional subspaces $X,Y \subset H$ are defined as follows. 
Let $\theta_{1}$ be the minimal value of $\arccos(|\langle x,y\rangle|)$ for unit vectors $x\in X$ and $y\in Y$.
Let also $x_{1}\in X$ and $y_{1}\in Y$ be unit vectors realizing this minimum.
For $i\ge 2$ the principal angle $\theta_{i}$ and unit vectors $x_{i}\in X$, $y_{i}\in Y$ are defined recursively, i.e.
$\theta_{i}$ is the minimal value of $\arccos(|\langle x,y\rangle|)$ for unit vectors $x\in X$ and $y\in Y$
orthogonal to $x_{1},\dots,x_{i-1}$ and $y_{1},\dots,y_{i-1}$ (respectively),
and $x_{i}\in X$, $y_{i}\in Y$ are unit vectors satisfying the latter conditions and realizing this minimum.

In  \cite{Molnar1,Molnar2} Moln\'ar  proposed the following extension of Wigner's theorem:
every (not necessarily bijective) transformation of ${\mathcal G}_{k}(H)$ preserving all principal angles between subspaces
is induced by a linear or conjugate-linear isometry on $H$
or $\dim H=2k$ and it is the composition of the transformation induced by an isometry and the orthocomplementation.
Recently, Geh\'er \cite{Geher2} obtained the same result for transformations of ${\mathcal G}_{k}(H)$
preserving the transition probability
(the transition probability is defined as the sum of squares of cosines for all principal angles).

Our main result states that the transformations of ${\mathcal G}_{k}(H)$ ($\dim H\ge 2k>2$)
induced by linear or conjugate-linear isometries can be characterized as 
transformations preserving some of principal angles
corresponding to the orthogonality, adjacency and ortho-adjacency relations.
To prove this statement, we use a modification of methods from \cite[Chapter 3]{Pankov-book}.

Gy\"ory \cite{Gyory} and \v{S}emrl \cite{Semrl} (see also \cite{GeherSemrl}) proved independently that 
every bijective transformation of ${\mathcal G}_{k}(H)$ preserving the orthogonality relation in both directions
is induced by an unitary or anti-unitary operator on $H$ under the assumption that $\dim H>2k$
(if $\dim H=2k$, then for every $X\in {\mathcal G}_{k}(H)$ the orthogonal complement $X^{\perp}$ is the unique element of ${\mathcal G}_{k}(H)$ orthogonal to $X$  and such transformations are wild).
It was observed in  \cite{Semrl} that 
there are non-bijective transformations of ${\mathcal G}_{k}(H)$ preserving the orthogonality relation in both directions 
which cannot be obtained from linear or conjugate-linear isometries.
As an application of the main result, we show that this happens only in the infinite-dimensional case.

\section{Results}
Let $X$ be a set and let $R\subset X\times X$ be a relation on $X$.  
We write $xRy$ if $(x,y)\in R$.
A transformation $f:X\to X$ is said to be $R$ {\it preserving} if for all $x,y\in X$ we have
$$xRy \Longrightarrow f(x)Rf(y);$$
in the case when 
$$xRy \Longleftrightarrow f(x)Rf(y)$$
for all $x,y\in X$, we say that $f$ is $R$ {\it preserving in both directions}.

It is clear that two elements of ${\mathcal G}_{k}(H)$ are orthogonal if and only if  all principal angles between them are equal to $\pi/2$.
We will always suppose that $\dim H\ge 2k$ (otherwise, there are no orthogonal pairs in ${\mathcal G}_{k}(H)$).

Two elements of ${\mathcal G}_{k}(H)$ are called {\it adjacency} if their intersection is $(k-1)$-dimensional, in other words,
only one of the principal angles between them is non-zero.
Two adjacent elements of ${\mathcal G}_{k}(H)$ are said to be {\it ortho-adjacency} 
if the unique non-zero principal angle between them is equal to $\pi/2$.

\begin{theorem}
Suppose that $\dim H>2k>2$.  
Let $f$ be an orthogonality preserving transformation of ${\mathcal G}_{k}(H)$ which satisfies one of the following additional conditions:
\begin{enumerate}
\item[(A)] $f$ is adjacency preserving.
\item[(OA)] $f$ is an ortho-adjacency preserving injection.
\end{enumerate}
Then $f$ is induced by a linear or conjugate-linear isometry on $H$.
\end{theorem}

For the case when $\dim H=2k$, we can prove only the following weak version of Theorem 1.

\begin{prop}
Suppose that $\dim H=2k>2$.  
Let $f$ be an orthogonality preserving transformation of ${\mathcal G}_{k}(H)$ 
which preserves the adjacency relation in both directions. 
Then $f$ is the bijection induced by an unitary or anti-unitary operator on $H$.
\end{prop}

We use Theorem 1 to prove the following.

\begin{theorem}
If the dimension of $H$ is finite and greater than $2k$,
then every transformation of ${\mathcal  G}_{k}(H)$ preserving the orthogonality relation in both directions 
is the bijective transformation of ${\mathcal G}_{k}(H)$ induced by an unitary or anti-unitary operator on $H$.
\end{theorem}

\begin{rem}{\rm
Chow's theorem \cite{Chow} describes 
bijective transformations of Grassmannians (of vector spaces) preserving the adjacency relation in both directions.
We refer \cite{PZ} for a description of bijections preserving the ortho-adjacency relation (for some reflexive forms) in both directions.
}\end{rem}

\section{Grassmann graph}
The {\it Grassmann graph} $\Gamma_{k}(H)$ is the graph whose vertex set is ${\mathcal G}_{k}(H)$
and whose edges are pairs of adjacent elements.
The case when $k=1$ is trivial  (any two distinct elements of ${\mathcal G}_{1}(H)$ are adjacent)
and we suppose that $k>1$.

For a subspace $S$ of dimension not greater than $k$ 
we denote by $[S\rangle_{k}$ the set of all $k$-dimensional subspaces containing $S$;
this set is called a {\it star} if  $\dim S=k-1$.
If $U$ is a subspace whose dimension is not less than $k$,
then we write $\langle U]_{k}$ for the set of all $k$-dimensional subspaces contained in $U$;
we say that this is a {\it top} if $\dim U=k+1$.

A {\it clique} in a graph is a set formed by mutually adjacent vertices. 
It is clear that stars and tops are cliques in $\Gamma_{k}(H)$.
Conversely, every maximal clique of $\Gamma_{k}(H)$ is a star or a top (see, for example, \cite{Pankov-book}).

Two subspaces $X,Y\subset H$ are called {\it compatible} if there are mutually orthogonal subspaces $X',Y',Z$
such that 
$$X=X'+Z\;\mbox{ and }\; Y=Y'+Z.$$
For example, any two incident or orthogonal subspaces are compatible. 
Two subspaces of $H$ are compatible if and only if there is an orthogonal basis of $H$
such that these subspaces are spanned by subsets of this basis.
Two elements of ${\mathcal G}_{k}(H)$ are ortho-compatible if they are adjacent and compatible. 

A subset of ${\mathcal G}_{k}(H)$ is called {\it compatible} if any two distinct elements from this subset are compatible. 
Every maximal compatible subset of ${\mathcal G}_{k}(H)$ is an {\it orthogonal apartment}, i.e.
it consists of all $k$-dimensional subspaces spanned by subsets of a certain orthogonal basis for $H$
\cite{Pankov}. 
Compatible subsets of cliques are formed by mutually ortho-adjacent elements.

\begin{lemma}\label{lemma0-0}
Every maximal compatible subset of a top contains precisely $k+1$ elements.
Every maximal compatible subset of a star contains precisely $n-k+1$ elements if $\dim H=n$ is finite,
and it is infinite if $H$ is infinite-dimensional.
\end{lemma}

\begin{proof}
Easy verification.
\end{proof}

The distance $d(v,w)$ between two vertices $v$ and $w$ in a connected graph 
is the smallest number of edges in a path connecting these vertices.
Every path between $v$ and $w$ which is formed by $d(v,w)$ edges is called a {\it geodesic}.
The Grassmann graph $\Gamma_{k}(H)$ is connected and the distance $d(X,Y)$
between $X,Y\in {\mathcal G}_{k}(H)$ in this graph is equal to $k-\dim(X\cap Y)$.
So, we have $d(X,Y)=k$ if and only if $X\cap Y=0$.
In particular, the distance between two orthogonal elements of ${\mathcal G}_{k}(H)$ is equal to $k$.

Let $X,Y\in {\mathcal G}_{k}(H)$ and 
$$X=X_{0},X_{1},\dots, X_{i}=Y,\;\;\;i=d(X,Y)$$
be a geodesic in $\Gamma_{k}(H)$.
An easy verification shows that for every $j\in \{1,\dots,i\}$ we have 
$$\dim(X\cap X_{j})=k-j\;\mbox{ and }\;\dim(Y\cap X_{j})=k-i+j.$$
Suppose that $X$ and $Y$ are orthogonal. 
Then $i=k$ and  $X_{j}$ is the orthogonal sum of $X\cap X_{j}$ and $Y\cap X_{j}$.
This means that $X_{j}$ is compatible to both $X$ and $Y$.
If $0<j<l<i$, then $X\cap X_{l}$ is contained in $X\cap X_{j}$, and $Y\cap X_{j}$ is contained in $Y\cap X_{l}$
which implies that $X_{j}$ and $X_{l}$ are compatible.

Conversely, let us consider  compatible $X,Y\in {\mathcal G}_{k}(H)$.
We take $Z\in {\mathcal G}_{k}(H)$ intersecting $Y$ precisely in $(X\cap Y)^{\perp}\cap Y$
and orthogonal to $X$.
Then $X,Y,Z$ are mutually compatible and there is an orthogonal apartment containing them.
This apartment contains  a geodesic joining $X$ with $Z$ and passing throught $Y$.

So, we get the following characterization of compatibility relation in terms of  orthogonality and adjacency.

\begin{lemma}\label{lemma0-1}
Every geodesic in $\Gamma_{k}(H)$ joining orthogonal elements consists of mutually compatible elements.
Any two compatible $X,Y\in {\mathcal G}_{k}(H)$ are contained in a certain geodesic of $\Gamma_{k}(H)$
connecting $X$ with an element orthogonal to $X$.
\end{lemma}

\section{Semilinear operators}
A mapping $L:H\to H$ is said to be a {\it semilinear operator} if 
$$L(x+y)=L(x)+L(y)$$
for all $x,y\in H$ and there is an endomorphism $\sigma$ of  the field ${\mathbb C}$ such that 
$$L(ax)=\sigma(a)L(x)$$
for all $a\in {\mathbb C}$ and all $x\in H$.
If an endomorphism of the field ${\mathbb C}$ is continuous, then it is identity or the conjugation. 
Non-continuous endomorphisms of ${\mathbb C}$ exist.
If a semilinear operator is bounded, then the associated endomorphism of ${\mathbb C}$ is continuous,
and the operator is linear or conjugate-linear.

Every injective semilinear operator on $H$ induces a transformation of ${\mathcal G}_{1}(H)$
and every non-zero scalar multiple of this operator defines the same transformation.
We will need the following consequence of the Fundamental Theorem of Projective Geometry
\cite{FF, Havlicek}.

\begin{fact}
Let $f$ be an injective transformation of ${\mathcal G}_{1}(H)$.
If for every $U\in {\mathcal G}_{2}(H)$ there is  $U'\in {\mathcal G}_{2}(H)$ such that 
$$f(\langle U]_{1})\subset \langle U']_{1}$$
and there is no $2$-dimensional subspace containing all elements from the image of $f$,
then $f$ is induced by an injective semilinear operator on $H$.
Such operator is unique up to a non-zero scalar multiple.
\end{fact}

\begin{rem}{\rm
The injectivity of $f$ cannot be omitted, see \cite[Example 2.3]{Pankov-book}.
}\end{rem}

Also, we will use the following.

\begin{lemma}\label{lemma0-2}
If an injective  semilinear operator on $H$ sends orthogonal vectors to orthogonal vectors,
then it is a non-zero scalar multiple of a linear or conjugate-linear isometry.
\end{lemma}

\begin{proof}
Let $L$ be a semilinear operator on $H$ sending orthogonal vectors to orthogonal vectors.
If $x,y\in H$ are orthogonal unit vectors, then $x+y,x-y$ are orthogonal. 
Since $L(x),L(y)$ and $L(x)+L(y),L(x)-L(y)$ are pairs of orthogonal vectors,
we have $$||L(x)||=||L(y)||.$$
If unit vectors $x,y\in H$ are non-orthogonal, then we choose a unit vector $z$ orthogo\-nal to both $x,y$ 
(this is possible, since $\dim H\ge 3$ by our assumption) and get 
$$||L(x)||=||L(z)||=||L(y)||.$$
So, the function $x\to ||L(x)||$ is constant on the set of unit vectors which means that $L$ is bounded,
i.e. $L$ is linear or conjugate-linear.

If $\{e_{i}\}_{i\in I}$ is an orthonormal basis of $H$, then  there is an orthonormal basis $\{e'_{i}\}_{i\in I}$  of $L(H)$ such that 
$L(e_{i})=a_{i}e'_{i}$ for non-zero scalars $a_{i}\in {\mathbb C}$.
We consider the orthogonal vectors $e_{i}+e_{j},e_{i}-e_{j}$  and establish that $|a_{i}|=|a_{j}|$ for any pair $i,j\in I$.
This implies the existence of a positive real number $b$ such that for every $i\in I$ we have $a_{i}=bb_{i}$, where $b_{i}$ is a unit complex number.
The linear or conjugate-linear operator transferring every $e_{i}$ to $b_{i}e'_{i}$ is an isometry.
We have $L=bL'$, where $L'$ is one of these operators. 
\end{proof}

\section{Proof of Theorem 1 and Proposition 1}
Let $f$ be an orthogonality preserving transformation of ${\mathcal G}_{k}(H)$ and $k>1$.

\subsection{The case (A)}
Suppose that $\dim H>2k$ and $f$ is adjacency preserving.

\begin{lemma}\label{lemma1-1}
The transformation $f$ is ortho-adjacency preserving.
\end{lemma}

\begin{proof}
Suppose that $X,Y\in {\mathcal G}_{k}(H)$ are ortho-adjacent.
Then $f(X),f(Y)$ are adjacent and we need to show that they are compatible.
By Lemma \ref{lemma0-1}, $X$ and $Y$ are contained in a certain geodesic $\gamma$ of $\Gamma_{k}(H)$ 
which connects $X$ with an element $Z\in {\mathcal G}_{k}(H)$ orthogonal to $X$.
Since $f$ is adjacency preserving, $f(\gamma)$ is a path in $\Gamma_{k}(H)$.
The elements $X$ and $Z$ are orthogonal,  and the same holds for $f(X)$ and $f(Z)$.
Then
$$d(X,Z)=d(f(X),f(Z))=k$$
which implies that $f(\gamma)$ is a geodesic  in $\Gamma_{k}(H)$ connecting $f(X)$ with $f(Z)$ and containing $f(Y)$.
Lemma \ref{lemma0-1} gives the claim.
\end{proof}

\begin{lemma}\label{lemma1-2}
For every star ${\mathcal S}\subset {\mathcal G}_{k}(H)$ there is the unique star containing $f({\mathcal S})$.
\end{lemma}

\begin{proof}
Since $f$ is adjacency preserving, $f({\mathcal S})$ is a clique in $\Gamma_{k}(H)$ (not necessarily maximal),
and it is contained in a certain maximal clique (a star or a top).
Let ${\mathcal X}$ be a maximal compatible subset of ${\mathcal S}$.
By Lemma \ref{lemma1-1},  $f({\mathcal X})$ is a compatible subset in a star or a top. 
Lemma \ref{lemma0-0} shows that $f({\mathcal X})$ cannot be contained in a top
(since $\dim H>2k$).
Therefore, $f({\mathcal S})$ is a subset in a star. 
The intersection of two distinct stars contains at most one element.
This means that there is the unique star containing $f({\mathcal S})$.
\end{proof}

Therefore, $f$ induces a transformation $f_{k-1}$ of ${\mathcal G}_{k-1}(H)$ such that
$$f([S\rangle_{k})\subset [f_{k-1}(S)\rangle_{k}$$
for every $S\in {\mathcal G}_{k-1}(H)$.
Then
$$f_{k-1}(\langle X]_{k-1})\subset \langle f(X)]_{k-1}$$
for every $X\in {\mathcal G}_{k}(H)$.

\begin{lemma}\label{lemma1-3}
The transformation $f_{k-1}$ is orthogonality preserving.
\end{lemma}

\begin{proof}
If $X$ and $Y$ are orthogonal elements of ${\mathcal G}_{k-1}(H)$,
then there exist orthogonal $X',Y'\in {\mathcal G}_{k}(H)$ containing $X$ and $Y$, respectively.
We have
$$f_{k-1}(X)\subset f(X'),\;f_{k-1}(Y)\subset f(Y')\;\mbox{ and }\; f(X')\perp f(Y')$$
which implies that $f_{k-1}(X)$ and $f_{k-1}(Y)$ are orthogonal.
\end{proof}

\begin{lemma}\label{lemma1-4}
The following assertions are fulfilled: 
\begin{enumerate}
\item[(i)]  If $X,Y\in {\mathcal G}_{k-1}(H)$ are adjacent, then $f_{k-1}(X)$ and $f_{k-1}(Y)$ are adjacent or coincident.
\item[(ii)]  $f_{k-1}$ is ortho-adjacency preserving.
\end{enumerate}
\end{lemma}

\begin{proof}
The statements are trivial for $k=2$.
Indeed, any two distinct elements of ${\mathcal G}_{1}(H)$ are adjacent,
and elements of ${\mathcal G}_{1}(H)$ are ortho-adjacent if they are orthogonal.
Consider the case when $k>2$.

(i)
If $X,Y\in {\mathcal G}_{k-1}(H)$ are adjacent, 
then the corresponding stars $[X\rangle_{k}$ and $[Y\rangle_{k}$ have a non-empty intersection. 
The transformation $f$ sends these stars to subsets of the same star or 
subsets of distinct stars with a non-empty intersection. 
This gives the claim.

(ii) 
By (i), $f_{k-1}$ sends every path in $\Gamma_{k-1}(H)$ to a path (possibly of shorter length).
So, the proof (ii) is similar to the proof of Lemma \ref{lemma1-1}.
\iffalse
Suppose that $X,Y\in {\mathcal G}_{k-1}(H)$ are ortho-adjacent.
By Lemma \ref{lemma0-1}, they are contained in a certain geodesic $\gamma$ of $\Gamma_{k-1}(H)$ 
which connects $X$ with an element $Z\in {\mathcal G}_{k-1}(H)$ orthogonal to $X$.
The statement (i) shows that $f(\gamma)$ is a path in $\Gamma_{k-1}(H)$.
Since $X$ and $Z$ are orthogonal,  the same holds for $f_{k-1}(X)$ and $f_{k-1}(Z)$ (Lemma \ref{lemma1-3})
and
$$d(X,Z)=d(f_{k-1}(X),f_{k-1}(Z))=k-1.$$
Then $f(\gamma)$ is a geodesic connecting $f_{k-1}(X)$ with $f_{k-1}(Z)$ and containing $f_{k-1}(Y)$.
Lemma \ref{lemma0-1} implies that $f_{k-1}(X)$ and $f_{k-1}(Y)$ are compatible, i.e.
they are ortho-adjacent by the statement (i).
\fi
\end{proof}

In the case when $k\ge 3$,
we use Lemma \ref{lemma1-4} and arguments from the proof of Lemma \ref{lemma1-2} to show that
for every star ${\mathcal S}\subset {\mathcal G}_{k-1}(H)$  there is the unique star containing $f_{k-1}({\mathcal S})$.

Step by step, we construct a sequence $f=f_{k},f_{k-1},\dots, f_{1}$,
where every $f_{i}$ is an orthogonality and ortho-adjacency preserving transformation of ${\mathcal G}_{i}(H)$.
If $i\ge 2$, then we have
$$f_{i}([Y\rangle_{i})\subset [f_{i-1}(Y)\rangle_{i}$$
for every $Y\in {\mathcal G}_{i-1}(H)$ and
$$f_{i-1}(\langle X]_{i-1})\subset \langle f(X)]_{i-1}$$
for every $X\in {\mathcal G}_{i}(H)$.
This implies that 
\begin{equation}\label{eq1-1}
f_{1}(\langle X]_{1})\subset \langle f(X)]_{1}\;\mbox{ if }\;X\in {\mathcal G}_{k}(H).
\end{equation}

\begin{lemma}\label{lemma1-8}
The transformation $f_{1}$ is injective.
\end{lemma}

\begin{proof}
For any distinct $P,Q\in {\mathcal G}_{1}(H)$ there exist mutually orthogonal $P_{1},\dots, P_{k-1}\in {\mathcal G}_{1}(H)$ which are orthogonal to both $P,Q$.
Consider the $k$-dimensional subspaces 
$$X=P_{1}+\dots +P_{k-1}+P\;\mbox{ and }\;Y=P_{1}+\dots+P_{k-1}+Q.$$
Since $f_{1}(P_{1}),\dots,f_{1}(P_{k-1}),f_{1}(P)$ are mutually orthogonal, \eqref{eq1-1} implies that 
$$f(X)=f_{1}(P_{1})+\dots +f_{1}(P_{k-1})+f_{1}(P).$$
Similarly, we establish that
$$f(Y)=f_{1}(P_{1})+\dots+f_{1}(P_{k-1})+f_{1}(Q).$$
The equality $f_{1}(P)=f_{1}(Q)$ implies that $f(X)=f(Y)$.
On the other hand, $X,Y$ are adjacent and the same holds for $f(X),f(Y)$.
\end{proof}

So, $f_{1}$ is an orthogonality preserving injective transformation of ${\mathcal G}_{1}(H)$
such that 
$$f_{1}(\langle Y]_{1})\subset \langle f_{2}(Y)]_{1}$$
for every $Y\in {\mathcal G}_{2}(H)$.
This means that  $f_{1}$ satisfies the conditions of the Fundamental Theorem of Projective Geometry (Fact 1), i.e. 
$f_{1}$ is induced by an injective semilinear operator on $H$.
This operator sends orthogonal vectors to orthogonal vectors
and Lemma \ref{lemma0-2} implies that it is a non-zero scalar multiple of 
a linear or conjugate-linear isometry. 
Using \eqref{eq1-1}, we show that this isometry induces $f$.

\subsection{The case (OA)}
Suppose that $\dim H>2k$ and $f$ is an ortho-adjacency preserving injection.
By Subsection 5.1, the required statement is a direct consequence of the following.

\begin{lemma}\label{lemma1-9}
The transformation $f$ is adjacency preserving.
\end{lemma}

\begin{proof}
If $X,Y\in {\mathcal G}_{k}(H)$ are adjacent, then $\dim(X+Y)=k+1$ and
$$\dim(X+Y)^{\perp}\ge 2$$
(since $\dim H >2k>2$). This implies the existence of orthogonal $P,Q\in {\mathcal G}_{1}(H)$ contained in $(X+Y)^{\perp}$.
We take
$$X'=(X\cap Y)+P\;\mbox{ and }\; Y'=(X\cap Y)+Q.$$
Then $X,X',Y'$ are mutually ortho-adjacent and the same holds for $Y,X',Y'$.
Let ${\mathcal X}$ be a maximal compatible subset of the star $[X\cap Y\rangle_{k}$ containing $X,X',Y'$.
Then $f({\mathcal X})$ is a compatible subset in a star or a top.
Since $\dim H>2k$, Lemma \ref{lemma0-0} implies that $f({\mathcal X})$ cannot be contained in a top,
i.e. it is a subset of a star. 
This means that $f(X)$ contains the $(k-1)$-dimensional subspace $f(X')\cap f(Y')$.
Similarly, we show that this subspace is contained in $f(Y)$.
Since $f$ is injective, $f(X)$ and $f(Y)$ are adjacent.
\end{proof}

\subsection{Proof of Proposition 1}
Suppose that $\dim H=2k$ and $f$ preserves the adjacency relation in both directions. 
%Then $f$ is ortho-adjacency preserving (it is clear that Lemma \ref{lemma1-1} holds for the case when $\dim H=2k$).
By  \cite[Section 3.3]{Pankov-book}, for every star ${\mathcal S}\subset {\mathcal G}_{k}(H)$
there is the unique maximal clique (a star or a top) containing $f({\mathcal S})$, and one of the following possibilities is realized:
\begin{enumerate}
\item[(S)] all stars go to subsets of stars,
\item[(T)] all stars go to subsets of tops.
\end{enumerate}
In the case (S), $f$ is induced by an unitary or anti-unitary operator on $H$
(arguments from Subsection 5.1).

In the case (T), we consider the composition of $f$ and the orthocomplementation. 
This transformation satisfies (S), i.e. it is induced by an unitary or anti-unitary operator.
This gives the claim.

\section{Proof of Theorem 2}

Let $f$ be  a transformation of ${\mathcal G}_{k}(H)$ preserving the orthogonality relation in both directions.
We suppose that $\dim H=n$ is finite and greater than $2k$.

\begin{lemma}\label{lemma2-1}
The transformation $f$ is injective. 
\end{lemma}

\begin{proof}
For distinct $X,Y\in {\mathcal G}_{k}(H)$ we take $Z \in {\mathcal G}_{k}(H)$ orthogonal to $X$ and non-orthogonal to $Y$.
Then $f(Z)$ is orthogonal to $f(X)$ and non-orthogonal to $f(Y)$.
This implies that $f(X)$ and $f(Y)$ are distinct.
\end{proof}

Consider the case when $k=1$.
For every $U\in {\mathcal G}_{2}(H)$ we take mutually orthogonal $Q_{1},\dots,Q_{n-2}\in  {\mathcal G}_{1}(H)$ contained in $U^{\perp}$. 
If $P\in \langle U]_{1}$, then $f(P)$ is contained in the orthogonal complement of 
the $(n-2)$-dimensional subspace 
$$f(Q_{1})+\dots+f(Q_{n-2}).$$
The images of three mutually orthogonal elements of ${\mathcal G}_{1}(H)$ are mutually orthogonal.
So, $f$ satisfies the conditions of the Fundamental Theorem of Projective Geometry (Fact 1).
Since $H$ is finite-dimensional, Lemma \ref{lemma0-2} shows that $f$ is induced by an unitary or anti-unitary operator.

From this moment we will suppose that $k>1$.

\begin{lemma}\label{lemma2-2}
Let $X_{1},\dots X_{i},Y$ be mutually distinct elements of ${\mathcal G}_{k}(H)$ such that 
$Y$ is not contained in $X_{1}+\dots +X_{i}$ and 
\begin{equation}\label{eq2-1}
\dim (X_{1}+\dots +X_{i})\le n-k.
\end{equation}
Then $f(Y)$ is not contained in $f(X_{1})+\dots +f(X_{i})$.
\end{lemma}

\begin{proof}
The condition \eqref{eq2-1} implies the existence of elements of ${\mathcal G}_{k}(H)$ orthogonal to $X_{1}+\dots +X_{i}$.
Since $Y$ is not contained in $X_{1}+\dots +X_{i}$, 
there is $Z\in{\mathcal G}_{k}(H)$ orthogonal to $X_{1}+\dots +X_{i}$ and non-orthogonal to $Y$.
Then $f(Z)$ is orthogonal to $f(X_{1})+\dots +f(X_{i})$ and non-orthogonal to $f(Y)$.
This implies the required.
\end{proof}

\begin{lemma}\label{lemma2-3}
The transformation $f$ is adjacency preserving.
\end{lemma}

\begin{proof}
Let $X$ and $Y$ be adjacent elements of ${\mathcal G}_{k}(H)$. 
Consider a sequence $$X_{0},X_{1},\dots,X_{n-2k}$$ of elements from ${\mathcal G}_{k}(H)$ such that $X_{0}=X$, $X_{1}=Y$ and 
$$\dim (X_{0}+X_{1}+\dots +X_{j})=k+j$$ 
for every $j\in \{1,\dots,n-2k\}$.
Then $X_{j}$ is not contained in $X_{0}+\dots+X_{j-1}$.
We have $k+j\le n-k$ for every $j\in \{1,\dots,n-2k\}$ and Lemma \ref{lemma2-2} implies that 
$$f(X_{j})\not\subset f(X_{0})+\dots +f(X_{j-1}).$$
Therefore,
\begin{equation}\label{eq2-2}
\dim (f(X_{0})+f(X_{1})+\dots+f(X_{n-2k}))\ge \dim(f(X_{0})+f(X_{1}))+n-2k-1.
\end{equation}
If $f(X)=f(X_{0})$ and $f(Y)=f(X_{1})$ are not adjacent, then 
$$\dim(f(X_{0})+f(X_{1}))>k+1$$
and \eqref{eq2-2} shows that
$$\dim (f(X_{0})+f(X_{1})+\dots+f(X_{n-2k}))>n-k.$$
In this case, there is no element of ${\mathcal G}_{k}(H)$ orthogonal to all $f(X_{j})$.
On the other hand, the equality 
$$\dim (X_{0}+X_{1}+\dots +X_{n-2k})=n-k$$
implies the existence of $Z\in {\mathcal G}_{k}(H)$ orthogonal to all $X_{j}$.
Then $f(Z)$ is orthogonal to every $f(X_{j})$.
This contradiction shows that $f(X)$ and $f(Y)$ are adjacent.
\end{proof}

The statement is a consequence of Lemma \ref{lemma2-3} and Theorem 1.

\end{document}